\documentclass[a4paper]{amsart}
\usepackage{graphicx}
\usepackage{amssymb}
\usepackage{amsmath}
\usepackage{amsthm}
\usepackage{amscd}
\usepackage[all,2cell]{xy}

\UseAllTwocells \SilentMatrices
\newtheorem{thm}{Theorem}[section]

\newtheorem{lem}[thm]{Lemma}
\newtheorem{exm}[thm]{Example}

\newtheorem{prop}[thm]{Proposition}
\theoremstyle{definition}
\newtheorem{defn}[thm]{Definition}
\theoremstyle{remark}

\numberwithin{equation}{section}

\begin{document}
\title[Totally reflexive extensions and modules]{Totally reflexive extensions and modules}

\author[  Xiao-Wu Chen
] {Xiao-Wu Chen}

\thanks{The author is supported by the Fundamental Research Funds for the Central Universities (WK0010000024), and National Natural Science Foundation of China (No.10971206).}
\date{\today}

\thanks{E-mail:
xwchen$\symbol{64}$mail.ustc.edu.cn}

\maketitle

\dedicatory{}%
\commby{}%

\begin{abstract}
We introduce the notion of totally reflexive extension of rings. It unifies Gorenstein orders and Frobenius extensions. We prove that for a totally reflexive extension, a module over the extension ring is totally reflexive if and only if its underlying module over the base ring is totally reflexive.
\end{abstract}

\section{Introduction}

The study of totally reflexive modules goes back to Auslander \cite{Au, ABr}, and is highlighted by
the work of Buchweitz on Tate cohomology of Gorenstein rings \cite{Buc}. Totally reflexive modules play
an important role in singularity theory \cite{Buc, Yoh}, cohomology theory of commutative rings \cite{AM, CFH}
and representation theory of artin algebras \cite{AR, RZ}. There are several different terminologies in the literature for these modules, such as modules of G-dimension zero \cite{ABr}, maximal Cohen-Macaulay modules \cite{Buc} and (finitely generated) Gorenstein-projective modules \cite{EJ1, Ch3}.

Recall that a ring extension means a homomorphism  $\theta\colon S\rightarrow R$ between two rings; thus $S$ is the base ring and
$R$ is the extension ring. We are interested in the following question.

What kind of ring extensions $\theta\colon S\rightarrow R$ satisfy the following condition: any $R$-module $X$ is totally reflexive if and only if the underlying $S$-module $X$ is totally reflexive?

We recall two well-known examples. Some results in \cite[Chapter III]{Aus} and \cite[Section 7]{Buc} suggest that a ring extension $\theta\colon S\rightarrow R$  satisfies the above condition, provided that $S$ is a commutative Gorenstein  ring and $R$ is a Gorenstein $S$-order. This generalizes the following classical example. Let $S=\mathbb{Z}$, and let $R=\mathbb{Z}G$ be the integral group ring of a finite group $G$. Then a $\mathbb{Z}G$-module $X$ is  totally reflexive if and only if the underlying $\mathbb{Z}$-module $X$ is totally reflexive, or equivalently, the abelian group $X$ is  free of finite rank; compare \cite[Section 8]{Buc}.

 For another example, let $S=kQ$ be the path algebra of a finite acyclic quiver $Q$ over a field $k$  and
 $A$ a finite dimensional Frobenius algebra. Set $R=S\otimes_k A$ to be the tensor product. Then the natural embedding $S\rightarrow R$ satisfies the above condition. Consequently, an $R$-module $X$ is totally reflexive if and only if the underlying $S$-module $X$ is totally reflexive, or equivalently, the corresponding representation $X$ of the quiver $Q$ is projective; compare \cite{Ch2, RZ, Wei}. Here, we recall the fact that $S=kQ$ is hereditary. We observe that this natural embedding $S\rightarrow R$  is a Frobenius extension  \cite{NT, Kad}. Note that the above classical example $\mathbb{Z}\rightarrow \mathbb{Z}G$ is also a Frobenius extension. So one expects that a Frobenius extension $\theta\colon S\rightarrow R$ might satisfy the above condition.

 We introduce the notion of totally reflexive extension. It is a ring extension $\theta\colon S\rightarrow R$ subject to two conditions: (1) the natural $S$-module  $R$ is totally reflexive on both sides; (2) there is an $R$-$S$-bimodule isomorphism ${\rm Hom}_S(R, P)\simeq W$ for some invertible $S$-bimodule $P$ and invertible $R$-bimodule $W$. Here, ${\rm Hom}_S(R, P)$ denotes
 the abelian group consisting of left $S$-module homomorphisms, which carries a natural $R$-$S$-bimodule structure.

  The new notion above unifies Gorenstein orders \cite{Aus} and Frobenius extensions \cite{NT}. Indeed, it unifies their generalizations such as Gorenstein algebras \cite{ABr} and Frobenius extensions of the third kind \cite{KS}.

The main result gives a partial answer to the above question.

\vskip 5pt

\noindent {\bf Theorem.} {\em Let $\theta\colon S\rightarrow R$ be a totally reflexive extension with respect to an invertible $S$-bimodule $P$ and an invertible $R$-bimodule $W$. Then an $R$-module $X$ is totally reflexive if and only if the underlying $S$-module $X$ is totally reflexive.}

\vskip 5pt

We mention that a related result is given in \cite{HS} for (weak) excellent extensions of rings. These extensions are not
totally reflexive extensions in general.

The paper is organized as follows. In Section 2, we introduce the notion of reflexive extension  and show that it is left-right symmetric; see Definition \ref{defn:1} and Proposition \ref{prop:symm}.  This notion is justified by a similar result as Theorem, where one replaces ``totally reflexive" by ``reflexive"; see Proposition \ref{prop:2}. In Section 3, we introduce the notion of totally reflexive extension, which is also left-right symmetric; see Definition \ref{defn:2}.  We point out that this notion unifies Gorenstein algebras and Frobenius extensions of the third kind; see Example \ref{exm:1}. Then we prove the main result; see Theorem \ref{thm:1}.

All rings in the paper are associative rings with a unit. Homomorphisms of rings are required to send the unit to the unit.

\section{Reflexive extensions}

In this section, we introduce the notion of reflexive extension of rings. This notion is justified by
the following result: in a reflexive extension, any module over the extension ring is reflexive if and
only the underlying module over the base ring is reflexive. The left-right symmetric property
of reflexive extensions is proved.

\subsection{Reflexive modules} Let $S$ be a ring. A left $S$-module $X$ is sometimes written as $_SX$.  We denote
by $S\mbox{-Mod}$ the category of left $S$-modules. For two
left $S$-modules $_SX$ and $_SY$, denote by ${\rm Hom}_S(X, Y)$ the abelian group consisting of
left $S$-homomorphisms between them. A right $S$-module $M$ is written as $M_S$. We identify right
$S$-modules as left $S^{\rm op}$-modules, where $S^{\rm op}$ is the opposite ring of $S$.
Hence, for two right $S$-modules $M_S$ and $N_S$, we denote by ${\rm Hom}_{S^{\rm op}}(M, N)$
the abelian group consisting of right $S$-homomorphisms between them. The $S$-action on modules is usually denoted by
``.".

Let $R$ and $T$ be another two rings. Assume that $_SX_R$ and $_SY_T$ are an $S$-$R$-bimodule and
an $S$-$T$-bimodule, respectively. Then the abelian group ${\rm Hom}_S(X, Y)$ is naturally an
$R$-$T$-bimodule in the following way: for $r\in R$, $f\in {\rm Hom}_S(X, Y)$ and $t\in T$,
$(r.f.t)(x)=f(x.r).t$ for all $x\in X$.

Let $P$ be an $S$-bimodule. For a left $S$-module $X$,  the \emph{evaluation map} with values in $P$
$${\rm ev}^P_X\colon X\longrightarrow {\rm Hom}_{S^{\rm op}}({\rm Hom}_S(X, P), P),$$
 defined as ${\rm ev}^P_X(x)(f)=f(x)$, is a left $S$-homomorphism. The module $_SX$ is $P$-\emph{reflexive} (resp. $P$-\emph{torsionless}), provided that
 ${\rm ev}^P_X$ is an isomorphism (resp. a monomorphism). In case that $P={_SS_S}$ is the regular bimodule, ${\rm ev}_X^P$ is abbreviated as ${\rm ev}_X$, and $P$-reflexive (resp. $P$-torsionless) modules are called  \emph{reflexive modules} (resp. \emph{torsionless modules}).  For example, a finitely generated projective $S$-module is reflexive. We have similar notions for right modules.

 We recall that an $S$-bimodule $P$ is \emph{invertible} provided that the endofunctor ${\rm Hom}_S(P, -)$ on the
 category $S\mbox{-Mod}$ of left $S$-modules is an equivalence. For example, the regular bimodule $_SS_S$ is
 invertible.

 The following fact is taken from  \cite[Chapter 12]{Fai73}.

 \begin{lem}\label{lem:inv}
 Let $P$ be an invertible $S$-bimodule. The following statements hold:
 \begin{enumerate}
 \item both $_SP$ and $P_S$ are finitely generated projective generators;
 \item  the natural ring homomorphisms $S\rightarrow {\rm End}_{S^{\rm op}}(P)$ and $S^{\rm op}\rightarrow {\rm End}_{S}(P)$,
 induced by the left $S$-action and right $S$-action, respectively, are isomorphisms;
 \item  both the endofunctors $-\otimes_SP$ and ${\rm Hom}_{S^{\rm op}}(P, -)$ on  the category $S^{\rm op}\mbox{-{\rm Mod}}$ of right $S$-modules are  equivalences. \hfill $\square$
 \end{enumerate}
 \end{lem}

The following result is well known.

\begin{lem}\label{lem:wellknown}
Let $P$ be an invertible $S$-bimodule. Then a left $S$-module $X$ is reflexive (resp. torsionless) if and only if it is $P$-reflexive ($P$-torsionless).
\end{lem}

\begin{proof}
Recall that the left $S$-module $_SP$ is finitely generated projective. Hence, the natural homomorphism
 \begin{align}\label{equ:2}
 t\colon {\rm Hom}_S(X, S)\otimes_SP\longrightarrow{\rm Hom}_S(X, P),\end{align}
 defined by $t(f\otimes p)(x)=f(x).p$, is an isomorphism. The following diagram commutes
\[\xymatrix{
X\ar@{=}[dd] \ar@{=}[r]^-{{\rm ev}_X} & {\rm Hom}_{S^{\rm op}}({\rm Hom}_S(X, S), S)\ar[d]^-{-\otimes_S P} \\
& {\rm Hom}_{S^{\rm op}}({\rm Hom}_S(X, S)\otimes_S P, P) \ar[d]^-{{\rm Hom}_{S^{\rm op}}(t, P)} \\
X \ar[r]^-{{\rm ev}_X^P} & {\rm Hom}_{S^{\rm op}}({\rm Hom}_S(X, P), P).
}\]
Here, we identify $P$ with $S\otimes_S P$. Recall that the endofunctor $-\otimes_SP$ on $S^{\rm op}\mbox{-Mod}$ is an equivalence. Then in the diagram above,
the map ${\rm ev}_X$ is an isomorphism (resp. a monomorphism) if and only if so is ${
\rm ev}_X^P$.
\end{proof}

\subsection{Reflexive extensions} Recall that a ring extension is a ring homomorphism $\theta\colon S\rightarrow R$.  Thus $S$ is the base ring and $R$ is the extension ring. We have that via $\theta$, $R$ becomes an $S$-bimodule. We sometimes consider $R$ as an $R$-$S$-bimodule or an $S$-$R$-bimodule. Similarly, an $R$-module $X$ is naturally an $S$-module, which is referred as  the \emph{underlying} $S$-module.

 \begin{defn}\label{defn:1} Let $\theta\colon S\rightarrow R$ be a ring extension. Let $P$ be an invertible $S$-bimodule and
 $W$ an invertible $R$-bimodule. The ring extension $\theta$ is  $P$-$W$-\emph{reflexive}, provided that the following two conditions are satisfied:
 \begin{enumerate}
 \item the left $S$-module $_SR$ is reflexive;
 \item there is an isomorphism ${\rm Hom}_S(R, P)\simeq W$ of $R$-$S$-bimodules.
 \end{enumerate}
 In case that $P={_SS_S}$ and $W={_RR_R}$ are the corresponding regular bimodules, a $P$-$W$-reflexive extension is called a \emph{reflexive extension}. \hfill $\square$
 \end{defn}

The following result implies that the above notion is left-right symmetric: a ring extension $\theta\colon S\rightarrow R$ is
$P$-$W$-reflexive if and only if the corresponding extension $\theta^{\rm op}\colon S^{\rm op}\rightarrow R^{\rm op}$ is $P$-$W$-reflexive. It is analogous to a classical result \cite[Theorem 1.2]{Kad} of Frobenius extensions.

\begin{prop}\label{prop:symm}
Let $\theta\colon S\rightarrow R$ be a ring extension. Assume that $P$ is an invertible $S$-bimodule and that $W$ is an invertible $R$-bimodule. Then the following statements are equivalent:
\begin{enumerate}
\item the ring extension $\theta$ is $P$-$W$-reflexive;
\item there is an $S$-bimodule homomorphism $\phi\colon W\rightarrow P$ such that the the induced maps
$$l_\phi\colon W\longrightarrow {\rm Hom}_S(R, P) \mbox{ and } r_\phi\colon {W} \longrightarrow {\rm Hom}_{S^{\rm op}}(R, P),$$
defined by $l_\phi(w)=\phi(-.w)$ and $r_\phi(w)=\phi(w.-)$, are bijective.
\item the right $S$-module $R_S$ is reflexive, and there is an isomorphism of $S$-$R$-bimodules ${\rm Hom}_{S^{\rm op}}(R, P) \simeq W$.
\end{enumerate}
\end{prop}

For each element $w\in W$, the left $S$-homomorphism $\phi(-.w)\colon R\rightarrow P$ is
defined as $\phi(-.w)(r)=\phi(r.w)$. Similarly, we have the right $S$-homomorphism $\phi(w.-)\colon R\rightarrow P$.

The $S$-bimodule homomorphism $\phi\colon W\rightarrow P$ is called the \emph{reflexive homomorphism} of the
ring extension $\theta$. We remark that the induced maps $l_\phi$ and $r_\phi$ are always an $R$-$S$-homomorphism and
an $S$-$R$-homomorphism, respectively.

\begin{proof}
``$(1)\Rightarrow (2)$" Denote by $\Phi\colon W\rightarrow {\rm Hom}_S(R, P)$ the isomorphism of $R$-$S$-bimodules
in the definition. We define $\phi\colon W\rightarrow P$ by $\phi(w)=\Phi(w)(1)$.
Observe that \begin{align}\label{equ:1}\phi(r.w)=\Phi(r.w)(1)=(r.\Phi(w))(1)=\Phi(w)(r),\end{align}
where the second equality uses that $\Phi$ is a left
$R$-homomorphism. This shows that $\Phi=l_\phi$; in particular, $l_\phi$ is bijective.

  We claim that $\phi$ is an $S$-bimodule homomorphism. Applying (\ref{equ:1}), we have $\phi(s.w)=\Phi(w)(\theta(s))=s.(\Phi(w)(1))$, since $\Phi(w)\colon R\rightarrow P$ is a left $S$-homomorphism. This proves that $\phi$ is a left $S$-homomorphism. It remains to show  $\phi(w.s)=\phi(w).s$. Recall that $\Phi(w.s)=\Phi(w).s$ for $s\in S$, since  $\Phi$ is a right $S$-homomorphism.  Then we have the following equalities \begin{align*}\phi(w.s)&=\Phi(w.s)(1)=(\Phi(w).s)(1)\\
                        &=\Phi(w)(1).s=\phi(w).s.\end{align*}

  It remains to show that $r_\phi$ is bijective. Observe that $r_\phi$ is a right $R$-homomorphism. Recall that $_SR$ is reflexive, and thus by Lemma \ref{lem:wellknown} the evaluation map ${\rm ev}_R^P$ is an isomorphism. Consider the following composite of isomorphisms
\begin{align*}
\Theta \colon R &\stackrel{{\rm ev}_R^P}\longrightarrow {\rm Hom}_{S^{\rm op}}({\rm Hom}_S(R, P), P)\stackrel{{\rm Hom}_{S^{\rm op}}(\Phi, P)}\longrightarrow {\rm Hom}_{S^{\rm op}}(W, P)\\
&\simeq {\rm Hom}_{S^{\rm op}}(W\otimes_R R, P) \simeq {\rm Hom}_{R^{\rm op}} (W, {\rm Hom}_{S^{\rm op}}(R, P)),
\end{align*}
where the last one is the adjoint isomorphism.  One computes that $(\Theta(r)(w))(r')=\phi(r.w.r')$.  It follows that
the following diagram commutes
\[\xymatrix{
R\ar[rr]\ar@{=}[d] && {\rm Hom}_{R^{\rm op}}(W, W)\ar[d]^-{{\rm Hom}_{R^{\rm op}}(W, r_\phi)}\\
R\ar[rr]^-{\Theta} &&  {\rm Hom}_{R^{\rm op}} (W, {\rm Hom}_{S^{\rm op}}(R, P)).
}\]
Here, the upper row is the isomorphism induced by the left $R$-action on the invertible $R$-bimodule $W$; see Lemma \ref{lem:inv}(2). It follows that that
${\rm Hom}_{R^{\rm op}}(W, r_\phi)$ is an isomorphism. Recall that ${\rm Hom}_{R^{\rm op}}(W, -)$ is an auto-equivalence
on $R^{\rm op}\mbox{-Mod}$; see Lemma \ref{lem:inv}(3). Then the homomorphism $r_\phi$ is an isomorphism.

``$(2)\Rightarrow (3)$" Recall that the map $r_\phi$ is a homomorphism of $S$-$R$-bimodules. Then we have the required isomorphism. It remains to show that the right $S$-module $R_S$ is reflexive.

Consider the composite of natural isomorphisms
\begin{align*}
\Xi \colon {\rm Hom}_S({\rm Hom}_{S^{\rm op}}(R, P), P)\stackrel{{\rm Hom}_S(r_\phi, P)}\longrightarrow &{\rm Hom}_S(W, P) \simeq {\rm Hom}_S(R\otimes_R W, P) \\
&\simeq {\rm Hom}_R(W, {\rm Hom}_S(R, P)),
\end{align*}
where the last one is the adjoint isomorphism. One computes that $(\Xi(f)(w))(r')=f(\phi(r'.w.-))$. Observe that $l_\phi$ is a left $R$-homomorphism, and thus an isomorphism.

The following diagram commutes by direct calculation
\[\xymatrix{
R\ar[d] \ar[rr]^-{{\rm ev}_R^P} && {\rm Hom}_S({\rm Hom}_{S^{\rm op}}(R, P), P) \ar[d]^-{\Xi}\\
{\rm Hom}_R(W, W) \ar[rr]^-{{\rm Hom}_R(W, l_\phi)} && {\rm Hom}_R(W, {\rm Hom}_S(R, P)), }\]
where the left vertical map is the isomorphism induced by the right $R$-action on the invertible bimodule $W$; see Lemma \ref{lem:inv}(2). This forces that
the evaluation map ${\rm ev}_R^P$ is an isomorphism. By Lemma \ref{lem:wellknown},  the right $S$-module $R_S$ is reflexive.

``$(3)\Rightarrow (1)$" Observe that (3) means exactly that the ring extension $\theta^{\rm op} \colon S^{\rm op}\rightarrow R^{\rm op}$ is $P$-$W$-reflexive. Then we obtain this implication by symmetry.
\end{proof}

We consider algebras over commutative rings. Let $S$ be a commutative ring. An $S$-\emph{algebra} is a ring extension
$\theta\colon S\rightarrow R$ such that the image lies in the center of $R$. In this case, ${\rm Hom}_S(R, S)$ is naturally an $R$-bimodule. We sometimes suppress $\theta$ and say that $R$ is an $S$-algebra.

We call an $S$-algebra $R$ is \emph{quasi-reflexive} provided that the $S$-module $_SR$ is reflexive and the $R$-bimodule ${\rm Hom}_S(R, S)$ is invertible; it is called \emph{reflexive}, provided that in addition there is an $R$-bimodule isomorphism ${\rm Hom}_S(R, S)\simeq R$.

\begin{lem}\label{lem:algebra}
Let $\theta\colon S\rightarrow R$ be an algebra such that $_SR$ is reflexive. Then the following statements are equivalent:
\begin{enumerate}
\item the ring extension $\theta$ is $S$-$W$-reflexive for some invertible $R$-module $W$;
\item the left $R$-module ${\rm Hom}_S(R, S)$ is a finitely generated projective generator;
\item the right $R$-module ${\rm Hom}_S(R, S)$ is a finitely generated projective generator.
\end{enumerate}
In this case, the $R$-bimodule ${\rm Hom}_S(R, S)$ is invertible.
\end{lem}

\begin{proof}
We only show the equivalence $(1)\Leftrightarrow (2)$. (1) implies (2), since an invertible bimodule
is a finitely generated projective generator on each side. On the other hand, observe that $W'={\rm Hom}_S(R, S)$
is an $R$-$R$-bimodule. Consider the ring homomorphism $\psi\colon R^{\rm op}\rightarrow {\rm End}_R(W')$ induced by the right $R$-action on $W'$. It suffices to show that it is an isomorphism. Then it follows from Morita theory that the $R$-module $W'$ is invertible, and
thus $\theta$ is $S$-$W'$-reflexive.

Consider the adjoint isomorphism $\gamma\colon {\rm Hom}_R(W', W')\stackrel{\sim}\longrightarrow {\rm Hom}_S({\rm Hom}_S(R, S), S)$. The following diagram commutes
\[\xymatrix{
R\ar@{=}[d]\ar[rr]^-{\psi} && {\rm Hom}_R(W', W')\ar[d]^-{\gamma}\\
R\ar[rr]^-{{\rm ev}_R} && {\rm Hom}_S({\rm Hom}_S(R, S), S).
}\]
By the assumption, ${\rm ev}_R$ is bijective. It follows that $\psi$ is bijective. We are done.
\end{proof}

\subsection{Preserving reflexivity} We will show that a reflexive extension $\theta\colon S\rightarrow R$ ``preserves reflexivity": an $R$-module $X$ is reflexive if and only if the underlying $S$-module $X$ is reflexive.

We need the following observation.

\begin{lem}\label{lem:phi}
Let $\theta\colon S\rightarrow R$ be a $P$-$W$-reflexive extension with $\phi\colon W\rightarrow P$ the
reflexive homomorphism.  Let $_RX$ and $Y_R$ be a left $R$-module and a right $R$-module, respectively. Then we have  the following:
\begin{enumerate}
\item there is an isomorphism of right $S$-modules $$t_X\colon{\rm Hom}_R(X, W)\stackrel{\sim}\longrightarrow {\rm Hom}_S(X, P),$$
sending $f$ to $\phi\circ f$;
\item there is an isomorphism of left $S$-modules
$$t_Y\colon {\rm Hom}_{R^{\rm op}}(Y, W)\stackrel{\sim}\longrightarrow {{\rm Hom}_{S^{\rm op}}(Y, P)},$$
sending $g$ to $\phi\circ g$.
\end{enumerate}
\end{lem}

\begin{proof}
(1) Recall that $l_\phi\colon W\rightarrow {\rm Hom}_S(R, P)$ is an isomorphism of $R$-$S$-bimodules. So we have the  isomorphism ${\rm Hom}_R(X, W)\simeq {\rm Hom}_R(X,  {\rm Hom}_S(R, P))$. Composing with an adjoint isomorphism, we have the required isomorphism $t_X$. A similar argument proves (2).
\end{proof}

The following result justifies the terminology ``reflexive extension".

\begin{prop}\label{prop:2}
Let $\theta\colon S\rightarrow R$ be a $P$-$W$-reflexive extension for  an invertible $S$-bimodule $P$ and an in vertible $R$-module $W$. Then a left $R$-module $_RX$ is reflexive (resp. torsionless) if and only if the underlying $S$-module $_SX$ is reflexive (resp. torsionless).
\end{prop}

\begin{proof} Denote by $\phi$ the reflexive homomorphism of the ring extension. Recall from Lemma \ref{lem:phi} the isomorphisms
$t_X$ and $t_{{\rm Hom}_R(X, W)}$. The following diagram commutes
\[\xymatrix{
X\ar@{=}[dd]\ar[r]^-{{\rm ev}_X^W} & {\rm Hom}_{R^{\rm op}}({\rm Hom}_R(X, W), W)\ar[d]^-{t_{{\rm Hom}_R(X, W)}}\\
                       &  {\rm Hom}_{S^{\rm op}}({\rm Hom}_R(X, W), P)\\
X \ar[r]^-{{\rm ev}_X^P}                       & {\rm Hom}_{S^{\rm op}}({\rm Hom}_S(X, P), P) \ar[u]_{{\rm Hom}_{S^{\rm op}}(t_X, P)}.
}\]
Indeed, both composite maps  send $x\in X$ to an element in ${\rm Hom}_{S^{\rm op}}({\rm Hom}_R(X, W), P)$, which maps
$g\in {\rm Hom}_R(X, W)$ to $\phi(g(x))$. It follows that the map ${\rm ev}_X^W$ is an isomorphism (resp. a monomorphism) if and only if so is ${\rm ev}_X^P$.  Applying Lemma \ref{lem:wellknown} twice, we are done with the proof.
\end{proof}

\section{Totally reflexive extensions}

In this section, we introduce the notion of totally reflexive extension and then prove Theorem. Before that, we
recall the notion of totally reflexive modules \cite{ABr, AM} and modules having finitely generated projective resolutions. We point out that totally reflexive extensions unify Gorenstein
algebras \cite{AH} and Frobenius extensions of the third kind \cite{KS}.

\subsection{Totally reflexive modules} Let $S$ a ring. Following \cite{AM}, a unbounded acyclic complex $P^\bullet=\cdots\rightarrow P^{-1}\rightarrow P^0\rightarrow P^1\rightarrow \cdots$ of $S$-modules is \emph{totally acyclic}, if each $P^i$ is finitely generated projective and the dual complex ${\rm Hom}_S(P^\bullet, S)$ is  also acyclic. A left $S$-module $X$ is \emph{totally reflexive} provided that there is a totally acyclic complex  $P^\bullet$ such that $X$ is isomorphic to its first cocycle $Z^1(P^\bullet)$. As is pointed out in the introduction, there are different terminologies for totally reflexive modules in the literature.

We denote by $S\mbox{-tref}$ the full subcategory of $S\mbox{-Mod}$ consisting of totally reflexive $S$-modules.
Observe that finitely generated projective modules are totally reflexive, and that totally reflexive modules
are reflexive. By \cite[Proposition 5.1]{AR} and its dual, the full subcategory $S\mbox{-tref}$ is closed under finite direct sums, direct summands and extensions.

 Recall that a left $S$-module $X$ is said to have  a \emph{finitely generated projective resolution} provided that there is a projective resolution $\cdots \rightarrow P^{-2}\rightarrow P^{-1}\rightarrow P^0\rightarrow X\rightarrow 0$ with each $P^{-i}$ finitely generated. We denote by $S\mbox{-fgpr}$ the full subcategory of $S\mbox{-Mod}$ consisting of such modules. Observe that $S\mbox{-tref}\subseteq S\mbox{-fgpr}$.

The following result is well known; compare \cite[(3.8)]{ABr} and \cite{EJ1}. In particular, the notion of
totally reflexive module is ``coordinate-free" (\cite[4.2]{Buc}).

\begin{lem}\label{lem:totally}
Let $P$ be an invertible $S$-bimodule. Then a left $S$-module $X$ is totally reflexive if and only if the following conditions are satisfied:
\begin{enumerate}
\item both $_SX$ and  ${\rm Hom}_S(X, P)$ have finitely generated projective resolutions;
\item $_SX$ is $P$-reflexive;
\item ${\rm Ext}^n_S(X, P)=0$ for $n\geq 1$;
\item ${\rm Ext}^n_{S^{\rm op}}({\rm Hom}_S(X, P), P)=0$ for $n\geq 1$.
\end{enumerate}
In particular, for a totally reflexive $S$-module $_SX$, the right $S$-module ${\rm Hom}_S(X, P)$  is also totally reflexive.
\end{lem}

\begin{proof}
Let $P^\bullet$ be an acyclic complex of finitely generated projective left $S$-modules. The isomorphism (\ref{equ:2}) induces an
isomorphism ${\rm Hom}_S(P^\bullet, P)\simeq {\rm Hom}_S(P^\bullet, S)\otimes_S P$ of complexes of right $S$-modules. Recall from Lemma \ref{lem:inv}(3) that the functor $-\otimes_S P$ is an auto-equivalence on $S^{\rm op}\mbox{-Mod}$. Then $P^\bullet$ is totally acyclic if and only if the complex ${\rm Hom}_S(P^\bullet, P)$ is acyclic.

For the ``only if" part, assume that $P^\bullet$ is a totally acyclic complex with $Z^1(P^\bullet)\simeq X$. Then the sub complex
$\cdots \rightarrow P^{-2} \rightarrow P^{-1}\rightarrow P^0\rightarrow 0 \rightarrow \cdots$ is a projective resolution of $X$. Hence the acyclicity of ${\rm Hom}_S(P^\bullet, P)$ implies  (3). It further implies that $\cdots \rightarrow {\rm Hom}_S(P^3, P) \rightarrow {\rm Hom}_S(P^2, P)\rightarrow {\rm Hom}_S(P^1, P)\rightarrow 0$ is a projective resolution of the right $S$-module ${\rm Hom}_S(X, P)$. Since each right $S$-module ${\rm Hom}_S(P^i, P)$ is finitely generated projective, we have (1). Applying ${\rm Hom}_{S^{\rm op}}(-, P)$ to this resolution, we deduce (2) and (4) from the acyclicity of $P^\bullet$; here, we use the isomorphism  $P^\bullet \simeq {\rm Hom}_{S^{\rm op}}({\rm Hom}_S(P^\bullet, P), P)$ of complexes, that is  induced by the evaluation maps ${\rm ev}_{P_i}^P$.

For the ``if" part, take projective resolutions $\cdots \rightarrow P^{-2}\rightarrow P^{-1}\rightarrow P^0\rightarrow X\rightarrow 0$ and $\cdots \rightarrow Q^{-2} \rightarrow Q^{-1}\rightarrow Q^{0}\rightarrow {\rm Hom}_{S}(X, P)\rightarrow 0$, where each $P^{-i}$ and $Q^{-j}$ are finitely generated. By (3), the complex $0\rightarrow {\rm Hom}_{S^{\rm op}}({\rm Hom}_S(X, P), P)\rightarrow P^1\rightarrow P^2\rightarrow P^3\rightarrow \cdots $ is acyclic, where $P^i={\rm Hom}_{S^{\rm op}}(Q^{-(i-1)}, P)$ are finitely generated projective for $i\geq 1$. Since $_SX$ is $P$-reflexive, we may splice the two complexes of left $S$-modules into an acyclic complex $P^\bullet$ of finitely generated projective left $S$-modules. Observe that $X\simeq Z^1(P^\bullet)$. By the discussion in the first paragraph, the complex $P^\bullet$ is totally acyclic.
\end{proof}

The following observation will be used in the next subsection.

\begin{lem}\label{lem:fgpr}
Let $S$ be a ring. Then the following statements hold:
\begin{enumerate}
\item the subcategory $S\mbox{-{\rm fgpr}}$ of $S\mbox{-{\rm Mod}}$ is closed under finite direct sums, direct summands and extensions;
\item an $S$-module $_SX$ lies in $S\mbox{-{\rm fgpr}}$ if and only if there is an acyclic complex $\cdots \rightarrow G^{-2}\rightarrow G^{-1}\rightarrow G^0\rightarrow X \rightarrow 0$ of $S$-modules such that each $G^{-i}$ lies in $S\mbox{-{\rm fgpr}}$.
\end{enumerate}
\end{lem}

\begin{proof}
(1) is a consequence of the dual of \cite[Proposition 5.1]{AR}. The ``only if" part of (2) is trivial.

For the ``if" part, denote by $\mathcal{G}$ the class of $S$-modules $X$, which fit into some acyclic complex  $\cdots \rightarrow G^{-2}\rightarrow G^{-1}\rightarrow G^0\rightarrow X \rightarrow 0$ with each $G^{-i}\in S\mbox{-{\rm fgpr}}$. It suffices to show that for each $X\in\mathcal{G}$, there is an exact sequence $0\rightarrow X'\rightarrow P^0 \rightarrow X\rightarrow 0$ with $_SP^0$ finitely generated projective and
$_SX'\in \mathcal{G}$.

Take an exact sequence $0\rightarrow X''\rightarrow G^{-1}\stackrel{f}\rightarrow G^0\rightarrow X\rightarrow 0$ with $G^{-i}\in S\mbox{-{\rm fgpr}}$ and $X''\in \mathcal{G}$. Since $G^0$ lies in $S\mbox{-{\rm fgpr}}$, we may take an exact sequence $0\rightarrow G'\rightarrow P^0\stackrel{g}\rightarrow G^0\rightarrow 0$ with $_SP^0$ finitely generated projective and $G'\in S\mbox{-{\rm fgpr}}$. Take the pullback of $f$ and $g$. We have the following commutative exact diagram
\[\xymatrix{
0\ar[r] & X'' \ar@{=}[d]\ar[r] & E \ar@{.>}[d] \ar@{.>}[r] & P^0 \ar[d]^-{g} \ar[r]^-{h} & X \ar@{=}[d] \ar[r] & 0\\
0\ar[r] & X'' \ar[r] & G^{-1} \ar[r]^-{f} & G^0 \ar[r] & X \ar[r] & 0.}\]
Observe the exact sequence $0\rightarrow G'\rightarrow E\rightarrow G^{-1}\rightarrow 0$ induced by the pullback. Applying (1), we have that $E$ lies in $S\mbox{-{\rm fgpr}}$. It follows that $X'={\rm Ker}\; h$ lies in $\mathcal{G}$. This gives us the required exact sequence. \end{proof}

\subsection{Totally reflexive extensions} We introduce the notion of totally reflexive extension. The main result claims that in a totally reflexive extension, a module over the extension ring is totally reflexive if and only if so is the underlying module over the base ring; see Theorem \ref{thm:1}.

\begin{defn}\label{defn:2}
Let $\theta\colon S\rightarrow R$ be a $P$-$W$-reflexive extension for  an invertible $S$-bimodule $P$ and an invertible $R$-bimodule $W$. The ring extension $\theta\colon S\rightarrow R$ is called \emph{totally $P$-$W$-reflexive} provided  in addition that the left $S$-module $_SR$ is  totally reflexive.\par
In case that $P={_SS_S}$ and $W={_RR_R}$ are the regular bimodules, a totally $P$-$W$-reflexive extension is called a \emph{totally reflexive extension}. \hfill $\square$
\end{defn}

The following observation implies that the notion of totally reflexive extension is left-right symmetric; compare  Proposition \ref{prop:symm}.

\begin{lem}
Let  $\theta\colon S\rightarrow R$ be a totally $P$-$W$-reflexive extension. Then the right $S$-module $R_S$ is
totally reflexive.
\end{lem}

\begin{proof}
Recall the isomorphism ${\rm Hom}_S(R, P)\simeq W$ of $R$-$S$-bimodules. By the assumption, $_SR$ is totally reflexive, and
thus by Lemma \ref{lem:totally}  so is the right $S$-module $W$. Recall that the subcategory $S^{\rm op}\mbox{-tref}$ is closed under finite direct sums and direct summands, and that the right $R$-module $W_R$ is a projective generator; see Lemma \ref{lem:inv}(1).  The right $S$-module $R_S$ is a direct summand of a finite direct sum of copies of $W_S$. It follows that  $R_S$ is totally reflexive.
\end{proof}

\begin{exm}\label{exm:1}
{\rm (1)} {\rm Let $S$ be a commutative ring and  $\theta\colon S\rightarrow R$ be an $S$-algebra. We call the $S$-algebra $R$ \emph{totally quasi-reflexive} provided that the ring extension $\theta$ is totally $S$-$W$-reflexive for an $R$-bimodule $W$; it is \emph{totally reflexive}, provided in addition that there is an isomorphism ${\rm Hom}_S(R, S)\simeq R$ of $R$-bimodules.

Let $S$ be a Gorenstein ring of finite Krull dimension. Then by Lemma \ref{lem:algebra}, the algebra $R$ is totally quasi-reflexive if and only if it  is a Gorenstein algebra in the sense of \cite[Definition 3.1]{AH}; the algebra $R$ is totally reflexive if and only if  it is a Gorenstein order in the sense of Auslander (\cite[Chapter III]{Aus}).}

{
{\rm (2)} \rm Let $\theta\colon S\rightarrow R$ be a $P$-$W$-reflexive extension. It is called \emph{$P$-$W$-Frobenius},  provided that the $S$-module $_SR$ is finitely generated projective, or equivalently, the right $S$-module $R_S$ is finitely generated projective. Observe that a $P$-$W$-Frobenius extension is totally $P$-$W$-reflexive.

 A $P$-$W$-Frobenius extension is called a \emph{Frobenius extension}, if $P=S$ and  $W=R$ are the regular bimodules. This classical notion goes to back to \cite{NT}. For example, the natural embedding $S\rightarrow SG$ of a ring $S$ into the group ring $SG$ of a finite group $G$ is a Frobenius extension.

We point out that a  $P$-$R$-Frobenius extension is exactly the $P$-Frobenius extension, or the Frobenius extension of the third kind in the sense of \cite[Definition 7.2]{KS}. \hfill $\square$}

\end{exm}

We now have the main result of the paper.

\begin{thm}\label{thm:1}
Let $\theta\colon S\rightarrow R$ be a totally $P$-$W$-reflexive extension for  an invertible $S$-bimodule $P$ and an invertible $R$-module $W$. Then a left $R$-module $_RX$ is totally reflexive if and only if the underlying $S$-module $_SX$ is totally reflexive.
\end{thm}

For the proof, we need the following general fact.

\begin{lem}\label{lem:fgpr2}
Let $\theta\colon S\rightarrow R$ be a ring extension with the left $S$-module $_SR\in S\mbox{-{\rm fgpr}}$. Then a left $R$-module $_RX$ lies in $R\mbox{-{\rm fgpr}}$ if and only if the underlying $S$-module $_SX$ lies in $S\mbox{-{\rm fgpr}}$.
\end{lem}

\begin{proof}
For the ``only if" part,take a projective $R$-resolution $\cdots \rightarrow P^{-2}\rightarrow P^{-1}\rightarrow P^0\rightarrow X\rightarrow 0$ with each $P^{-i}$ finitely generated. Recall that the subcategory $S\mbox{-{\rm fgpr}}$ is closed under finite direct sums and direct summands. From the assumption, we infer that each  $S$-module $P^{-i}$ lies in $S\mbox{-{\rm fgpr}}$. By Lemma \ref{lem:fgpr}(2), we have that  $_SX$ lies in $S\mbox{-{\rm fgpr}}$.

For the ``if" part, we assume that $_SX$ lies in $S\mbox{-{\rm fgpr}}$. In particular, the $R$-module $X$ is finitely generated.
Take an exact sequence $0\rightarrow X'\rightarrow P^0\stackrel{f}\rightarrow X\rightarrow 0$ of $R$-modules with $P^0$ finitely generated projective. Using induction, it suffices to show that the underlying $S$-module of $X'$ lies in $S\mbox{-{\rm fgpr}}$.

Take an exact sequence $0\rightarrow Y\rightarrow Q \stackrel{g}\rightarrow X\rightarrow 0$ of $S$-modules with $_SQ$ finitely generated projective and $_SY\in S\mbox{-{\rm fgpr}}$. Consider the following commutative exact diagram of $S$-modules, that is induced by the pullback of $f$ and $g$
\[\xymatrix{
0\ar[r] & X'\ar@{=}[d] \ar[r] &  E\ar@{.>}[r] \ar@{.>}[d] & Q\ar[d]^-{g} \ar[r] & 0\\
0\ar[r] & X' \ar[r] &  P^0\ar[r]^-{f} &  X \ar[r] & 0.
}\]
The upper row splits, and thus $E\simeq X'\oplus Q$. Consider the exact sequence $0\rightarrow Y\rightarrow E\rightarrow P^0\rightarrow 0$ induced by the pullback. From the assumption, we have that $_SP^0$ lies in $S\mbox{-{\rm fgpr}}$. Recall from Lemma \ref{lem:fgpr}(1) that  the subcategory $S\mbox{-{\rm fgpr}}$ is closed under extensions and direct summands.  We have that $_SE$ lies in $S\mbox{-{\rm fgpr}}$, and thus $_SX'$ also lies in $S\mbox{-{\rm fgpr}}$.
\end{proof}

We are in a position to prove Theorem \ref{thm:1}.

\vskip 5pt

\noindent \emph{Proof of Theorem} \ref{thm:1}. Recall from Lemma \ref{lem:totally} the  conditions (1)-(4)
for totally reflexive modules. Indeed, we will show that, for each  $1\leq i\leq 4$,  the $R$-module $_RX$, with respect to the invertible $R$-bimodule $W$, satisfies the condition ($i$) if and only if so does the underlying $S$-module $_SX$, with respect to the invertible $S$-bimodule $P$.

For  (1), recall the isomorphism in Lemma \ref{lem:phi}(1). Applying Lemma \ref{lem:fgpr2} and its counterpart for right modules, we have (1).

For  (2), apply Lemma \ref{lem:wellknown} and Proposition \ref{prop:2}.

For  (3), it suffices to show that there is an isomorphism
$${\rm Ext}^n_R(X, W)\simeq {\rm Ext}^n_S(X, P)$$
of right $S$-modules, for each $n\geq 1$.  Take a projective $R$-resolution $P^\bullet=\cdots \rightarrow P^{-2}\rightarrow P^{-1}\rightarrow P^0\rightarrow 0 \rightarrow \cdots$ of $X$. The isomorphism in Lemma \ref{lem:phi}(1) induces an isomorphism ${\rm Hom}_R(P^\bullet, W)\simeq {\rm Hom}_S(P^\bullet, P)$ of complexes. Hence, we have ${\rm Ext}^n_R(X, W)\simeq H^n({\rm Hom}_S(P^\bullet, P))$; here, $H^n(-)$ denotes the $n$-th cohomology of a complex. By assumption, the left $S$-module $_SR$ is totally reflexive. In particular, by Lemma \ref{lem:totally}(3) we have ${\rm Ext}_S^k(R, P)=0$ for $k\geq 1$. It follows that each projective $R$-module $P^{-i}$ satisfies that ${\rm Ext}_S^k(P^{-i}, P)=0$ for $k\geq 1$. By \cite[2.4.3]{Weib}, this implies that $H^n({\rm Hom}_S(P^\bullet, P))\simeq {\rm Ext}_S^n(X, P)$.

For (4),  it suffices to show that there is an isomorphism
$${\rm Ext}^n_{R^{\rm op}}({\rm Hom}_R(X, W), W)\simeq  {\rm Ext}^n_{S^{\rm op}}({\rm Hom}_S(X, P), P)$$
for each $n\geq 1$. A similar argument as above yields an isomorphism $${\rm Ext}^n_{R^{\rm op}}({\rm Hom}_R(X, W), W)\simeq {\rm Ext}^n_{S^{\rm op}}({\rm Hom}_R(X, W), P).$$
Applying the isomorphism in Lemma \ref{lem:phi}(1), we are done. \hfill $\square$

\vskip 5pt

Recall that the notion of Gorenstein-projective module is a natural extension of
totally reflexive module to unnecessarily finitely generated modules; see \cite{EJ1}. We expect that
the following result holds:  for a totally $P$-$W$-reflexive extension $\theta\colon S\rightarrow R$,
an $R$-module $X$ is Gorenstein-projective if and only if the underlying $S$-module $X$ is Gorenstein-projective.
It seems that a different argument is needed, since the present one can not carry over to Gorenstein-projective modules.

\vskip10pt

%
%
%
%
%
%

\noindent {\bf Acknowledgements.} \; This research was partly done during the author's visit at the University of Bielefeld with a support by Alexander von Humboldt Stiftung. He would like to thank Professor Henning Krause and the faculty of Fakult\"{a}t f\"{u}r Mathematik for their hospitality.

\bibliography{}

\vskip 10pt

 {\footnotesize \noindent Xiao-Wu Chen\\
  School of Mathematical Sciences, University of Science and Technology of
China, Hefei 230026, Anhui, PR China \\
Wu Wen-Tsun Key Laboratory of Mathematics, USTC, Chinese Academy of Sciences, Hefei 230026, Anhui, PR China.\\
URL: http://mail.ustc.edu.cn/$^\sim$xwchen}

\end{document}